\documentclass{article}
\usepackage{amssymb}
\usepackage{amsmath}

\newtheorem{theorem}{Theorem}

\newtheorem{corollary}[theorem]{Corollary}

\newtheorem{definition}[theorem]{Definition}

\newtheorem{lemma}[theorem]{Lemma}

\newtheorem{proposition}[theorem]{Proposition}

\newenvironment{proof}[1][Proof]{\noindent\textbf{#1.} }{\ \rule{0.5em}{0.5em}}
\input{tcilatex}

\begin{document}

\title{\textbf{Co-radiant set-valued mappings}}
\author{\textbf{Abelardo Jord\'{a}n} \\
%EndAName
Departamento Acad\'{e}mico de Ciencias - Secci\'{o}n Matem\'{a}ticas\\
Pontificia Universidad Cat\'{o}lica del Per\'{u}\\
Peru\\
ajordan@pucp.edu.pe \and \textbf{Juan Enrique Mart\'{\i}nez-Legaz} \\
%EndAName
Department d'Economia i d'Hist\`{o}ria Econ\`{o}mica\\
Universitat Aut\`{o}noma de Barcelona;\\
Barcelona Graduate School of Mathematics (BGSMath)\\
Spain\\
\texttt{JuanEnrique.Martinez.Legaz@uab.es} \and \textit{Dedicated to
Alexander Ioffe}}
\maketitle

\begin{abstract}
\textit{\noindent }We introduce and study a notion of co-radiantness for
set-valued mappings between nonnegative orthants of Euclidean spaces. We
analyze them from an abstract convexity perspective. Our main results
consist in representations, in terms of intersections of graphs, of the
increasing co-radiant mappings that take closed normal values, by means of
elementary members in the class of such mappings.\medskip

\medskip

\end{abstract}

\section{Introduction}

In this work we introduce and study a notion of co-radiantness for
set-valued mappings from $\mathbb{R}_{+}^{n}$ into $\mathbb{R}_{+}^{m}.$ The
main aim is to extend the results of \cite{Martinez} on single-valued
increasing co-radiant functions, especially their abstract convexity
representations by elementary functions. Such functions arise in economic
theory, where they are used to model single output production with
decreasing returns to scale \cite{I02}. Similarly, our co-radiant set-valued
vector mappings have obvious potential applications to model economies of
scale in the multiple output case, which have been considered in \cite{PW77}%
. Having such an application in mind, we will restrict our attention to
set-valued mappings taking only normal values, a quite natural property when
dealing with production technologies. Moving from the single-valued case to
a set-valued vector setting provides much more flexibility to model
production technologies. Our approach departs from the classical one in
production theory, which is mostly based on ordinary convexity, in that our
main tools belong to abstract convexity theory. We will assume the values of
our mappings to be normal, but not necessarily convex, thus making our
approach radically different from the one based on convex processes \cite%
{R67, N82}.

The results we present in this paper constitute nontrivial extensions of the
existing ones on single-valued functions. They belong to an abstract
convexity framework; our main contribution consists in identifying suitable
elementary mappings within the class of increasing co-radiant set-valued
mappings which generate the whole class when taking intersections of their
graphs. For simplicity, our developments are presented in a finite
dimensional setting, but extensions to real topological vector spaces, as
those in \cite{DM09} for the single-valued case, are certainly possible.

The structure of the paper is as follows. In Section 2 we present the
fundamental definitions and results our developments will be based on. We
recall there the main notions of abstract convexity and set-valued mappings
we need. In Section 3 we introduce and study co-radiant set-valued mappings.
Section 4 contains our main results on the representation of increasing
co-radiant set-valued mappings by means of suitably defined elementary
members in the class of such mappings.

\bigskip

\textbf{Notations}

\bigskip

We mostly follow the standard notation in set-valued analysis \cite{Aubin}
and abstract convexity theory \cite{Rubinov}. We denote by $\mathbb{R}%
_{+}^{n}$ the nonnegative orthant in $\mathbb{R}^{n}.$ For $x,y\in \mathbb{R}%
_{+}^{n},$ by $x\leq y$ we mean that $y-x\in \mathbb{R}_{+}^{n}.$ The
closure and the interior of a set $C$ are $\overline{C}$ and $int(C),$
respectively. We set $\mathbb{R}_{++}^{n}:=int\left( \mathbb{R}%
_{+}^{n}\right) .$ The domain and the graph of a set-valued mapping $F:%
\mathbb{R}_{+}^{n}\rightrightarrows \mathbb{R}_{+}^{m}$ are $dom(F):=\left\{
x\in \mathbb{R}_{+}^{n}:F\left( x\right) \neq \emptyset \right\} $ and $%
gr(F):=\left\{ \left( x,y\right) \in \mathbb{R}_{+}^{n}\times \mathbb{R}%
_{+}^{m}:y\in F\left( x\right) \right\} ,$ respectively. By $\left\Vert
\cdot \right\Vert _{\infty }$ we denote the maximum norm, defined on $%
\mathbb{R}^{n}$ by $\left\Vert x\right\Vert _{\infty
}:=\max_{i=1,...,n}\left\vert x_{i}\right\vert $ for $x:=\left(
x_{1},...,x_{n}\right) .$ Its associated unit ball is $\mathbf{B}_{\infty }.$

\section{Preliminaries}

Our notion of co-radiant set-valued function is a generalization of the
corresponding one for single-valued functions, which is in turn based on the
fundamental concepts of radiant and co-radiant sets.

\begin{definition}
\cite{RubinovDemyanov} A nonempty set $C\subset \mathbb{R}_{+}^{n}$ is called

\begin{itemize}
\item[(i)] $\mathbf{radiant}$, if $x\in C,t\in $ $(0,1]\Longrightarrow $ $%
tx\in C.$

\item[(ii)] $\mathbf{co}$\textbf{-}$\mathbf{radiant}$, if $x\in C,\lambda
\geq 1$ $\Longrightarrow $ $\lambda x\in C.$
\end{itemize}
\end{definition}

\begin{definition}
\label{ContieneFuncionInvCoradiant}\cite{Rubinov, Rubinov2} A function $f:%
\mathbb{R}_{+}^{n}\rightarrow \mathbb{R}_{+}\mathbb{\cup }\left\{ +\infty
\right\} $ is called $\mathbf{co}$\textbf{-}$\mathbf{radiant}$, if 
\begin{equation*}
f(tx)\geq tf(x),\hspace{1cm}\forall x\in \mathbb{R}_{+}^{n},t\in (0,1].
\end{equation*}
\end{definition}

It is easy to see that $f$ is co-radiant if, and only if,%
\begin{equation*}
f(\lambda x)\leq \lambda f(x),\forall x\in \mathbb{R}_{+}^{n},\lambda \geq 1.
\end{equation*}%
It is also a simple exercise to prove that a function $f\not\equiv +\infty $
is co-radiant if, and only if, the set $hypo(f):=\{(x,\alpha )\in \mathbb{R}%
_{+}^{n}\times \mathbb{R}_{+}:\,f(x)\geq \alpha \}$ is radiant. Co-radiant
functions with the additional property of being increasing have been
intensively studied in \cite{Rubinov2}, where applications to a class of
optimization problems are discussed.

\begin{definition}
\label{DefCoNormal}(see \cite[p. 106]{RG98}). A set $C\subset \mathbb{R}%
_{+}^{n}$ is called$\mathbf{\ normal}$ if it satisfies the following
property: 
\begin{equation*}
x\in C,\,0\leq y\leq x\Rightarrow y\in C.
\end{equation*}
\end{definition}

We next recall the basic properties of normal sets (see \cite{Tuy1}, \cite%
{Rubinov} and \cite{Rubinov_Singer}):

\begin{proposition}
\label{propiedades}\hfill {}
\end{proposition}

\begin{itemize}
\item[(i)] Both $\mathbb{R}_{+}^{n}$ and $\emptyset $ are normal sets.

\item[(ii)] For any $y\in \mathbb{R}_{+}^{n}$, the set $[0,y]:=\{x\in 
\mathbb{R}_{+}^{n}:\,x\leq y\}$ is normal.

\item[(iii)] If $\{C_{i}\}_{i\in I}$ is an arbitrary collection of normal
sets in $\mathbb{R}_{+}^{n}$, then $\displaystyle\dbigcup\limits_{i\in
I}C_{i}$ and $\displaystyle\dbigcap\limits_{i\in I}C_{i}$ are normal sets,
too.

\item[(iv)] If $C$ is a normal set in $\mathbb{R}_{+}^{n}$, then $\overline{C%
}$ is normal, too.

\item[(v)] If $C$ is a normal set in $\mathbb{R}_{+}^{n}$, then the
following equivalence holds true:%
\begin{equation*}
C\cap \mathbb{R}_{++}^{n}\neq \emptyset \Leftrightarrow int(C)\neq \emptyset
.
\end{equation*}
\end{itemize}

As in \cite{Rubinov_Singer}, we consider the coupling function $\left\langle
\cdot ,\cdot \right\rangle :$ $\mathbb{R}_{+}^{n}\times \mathbb{R}%
_{+}^{n}\rightarrow \mathbb{R}_{+}$ defined by 
\begin{equation*}
\langle \ell ,x\rangle :=\left\{ 
\begin{array}{ll}
\displaystyle\min_{i\in I_{+}(\ell )}\ell _{i}x_{i}, & \hbox{ if }\ell \neq 0
\\ 
0, & \hbox{ if }\ell =0,%
\end{array}%
\right.
\end{equation*}%
with $\ell =(\ell _{1},\cdots \ell _{n})$ and $I_{+}(\ell )=\left\{ i\in
\left\{ 1,\cdots ,n\right\} :\ell _{i}>0\right\} $. Functions of the type $%
\langle \ell ,\cdot \rangle $ are called $min-type$ functions. Their level
sets $\{x\in \mathbb{R}_{+}^{n}:\,\langle \ell ,x\rangle \leq \alpha \}$ ($%
\alpha \geq 0$) are closed normal sets. From the following proposition it
immediately follows that every closed normal set is the intersection of a
collection of such level sets.

\begin{proposition}
\label{SeparNormalCerr}\cite[Proposition 2.3]{Rubinov_Singer} For a subset $%
C $ of $\mathbb{R}_{+}^{n}$, the following conditions are equivalent:

\begin{itemize}
\item[(i)] $C$ is normal and closed.

\item[(ii)] For each $x\in \mathbb{R}_{+}^{n}\setminus C$ there exists $\ell
\in \mathbb{R}_{+}^{n}$ such that 
\begin{equation*}
\langle \ell ,y\rangle \leq 1<\langle \ell ,x\rangle ,\hspace{0.4cm}\forall
y\in C.
\end{equation*}
\end{itemize}
\end{proposition}

The folllowing notion of support function, which is useful for dealing with
closed normal sets, was essentially introduced in \cite[p. 108]{RG98}.

\begin{definition}
The\textbf{\ support function }of\textbf{\ }$C\subset \mathbb{R}_{+}^{n}$ is 
$\sigma _{C}$:$\mathbb{R}_{+}^{n}\rightarrow \mathbb{R}_{+}\cup \left\{
+\infty \right\} $, defined by 
\begin{equation*}
\displaystyle\sigma _{C}(\ell ):=\sup \{\langle \ell ,x\rangle \;:\;x\in C\},
\end{equation*}%
with the convention $\sup \emptyset =0.$
\end{definition}

We next present some basic properties of $\sigma _{C}$.

\begin{proposition}
\label{properties sigma}For every $C\subset \mathbb{R}_{+}^{n}$, one has:

\begin{itemize}
\item[(i)] $\sigma _{C}(0)=0.$

\item[(ii)] $\sigma _{C}$ is positively homogeneous, that is, for every $%
\alpha >0$ and $\ell \in \mathbb{R}_{+}^{n}$ one has $\sigma _{C}(\alpha
\ell )=\alpha \sigma _{C}(\ell ).$

\item[(iii)] $\sigma _{C}$ is increasing.

\item[(iv)] $\sigma _{C}$ is lower semicontinuous.

\item[(v)] $\displaystyle\sigma _{C}=\sigma _{\overline{C}}.$

\item[(vi)] If $\lambda >0$ is positive, then $\sigma _{\lambda C}=\lambda
\sigma _{C}.$

\item[(vii)] If $A$ and $B$ are nonempty subsets of $\mathbb{R}_{+}^{n}$,
then%
\begin{equation}
\sigma _{A+B}\geq \sigma _{A}+\sigma _{B}.  \label{sop suma}
\end{equation}

\item[(viii)] If $A$ and $B$ are subsets of $\mathbb{R}_{+}^{n}$ and $B$ is
normal and closed, the following equivalence holds true:%
\begin{equation*}
A\subset B\Leftrightarrow \sigma _{A}\leq \sigma _{B}.
\end{equation*}
\end{itemize}
\end{proposition}

\begin{proof}
Statements (i) (ii), (iii) and (vi) follow immediately from the definition
of $\sigma _{C}$. Statements (iv) and (v) are easy consequences of the
continuity of the min-type functions $\langle \ell ,\cdot \rangle $.
Statement (vii) follows from the superadditivity of the min-type functions $%
\langle \ell ,\cdot \rangle $. The implication $\Longrightarrow $ in (viii)
is obvious. Conversely, assume that $\sigma _{A}\leq \sigma _{B}.$ Then, if $%
x\in \mathbb{R}_{+}^{n}\setminus B,$ by Proposition \ref{SeparNormalCerr}
there exists $\ell \in \mathbb{R}_{+}^{n}$ such that $\langle \ell ,y\rangle
\leq 1<\langle \ell ,x\rangle $ for every $y\in B;\ $hence $\sigma _{A}(\ell
)\leq \sigma _{B}(\ell )\leq 1<\langle \ell ,x\rangle $, which shows that $%
x\notin A.$ This proves the inclusion $A\subset B.$
\end{proof}

\bigskip

We observe that equality does not necessarily hold in (\ref{sop suma}), even
if $A$ and $B$ are normal. Consider, for instance, in $\mathbb{R}_{+}^{2}$
the segments $A$ and $B$ joining the origin with the points $(1,0)$ and $%
(0,1),$ respectively; for $\ell :=(1,1)$ one has $\displaystyle\sigma
_{A}(\ell )=0,\,\sigma _{B}(\ell )=0$ and $\sigma _{A+B}(\ell )=1.$

\begin{corollary}
If $A$ and $B$ are normal and closed subsets of $\mathbb{R}_{+}^{n},$ the
following equivalence holds true:%
\begin{equation*}
A=B\Leftrightarrow \sigma _{A}=\sigma _{B}.
\end{equation*}
\end{corollary}

\begin{proof}
It follows from Proposition \ref{properties sigma}(viii).
\end{proof}

\begin{corollary}
\label{RepresSoporte}A set $C\subset \mathbb{R}_{+}^{n}$ is a normal and
closed if, and only if,%
\begin{equation}
\displaystyle C=\{\,x\in \mathbb{R}_{+}^{n}\,:\langle \ell ,x\rangle \leq
\sigma _{C}(\ell ),\hspace{0.3cm}\forall \ell \in \mathbb{R}_{+}^{n}\,\}.
\label{normal from support}
\end{equation}
\end{corollary}

\begin{proof}
The "if" statement is an immediate consequence of Proposition \ref%
{properties sigma}(iv) (closedness of $C)$ together with the fact that the
left hand sides of the inequalities in (\ref{normal from support}) are
increasing (normality of $C).$ To prove the converse, let $D\ $be the right
hand side of (\ref{normal from support}). It is immediate that $C\subset D$
and $\sigma _{D}\leq \sigma _{C}$; from this inequality and Proposition \ref%
{properties sigma}(viii), one gets the opposite inclusion $D\subset C.$
\end{proof}

\section{Increasing Co-radiant (ICR) Set-Valued Mappings}

The notion of increasingness we will use for set-valued mappings is provided
by the following definition.

\begin{definition}
A mapping $F:\mathbb{R}_{+}^{n}\rightrightarrows \mathbb{R}_{+}^{m}$ is
called\textbf{\ increasing} if 
\begin{equation*}
x\leq y\Rightarrow F(x)\subset F(y).
\end{equation*}
\end{definition}

We next extend the notion of co-radiant function \cite{Rubinov} to
set-valued mappings.

\begin{definition}
A mapping $F:\mathbb{R}_{+}^{n}\rightrightarrows \mathbb{R}_{+}^{m}$ with
nonempty graph will be called \textbf{co-radiant}, if%
\begin{equation*}
F(tx)\supset tF(x),\hspace{0.5cm}\forall x\in \mathbb{R}_{+}^{n},t\in (0,1].
\end{equation*}%
\label{casicorad}
\end{definition}

Let us observe that condition (\ref{strongcorad}) can be equivalently
written as 
\begin{equation}
F(\lambda x)\subset \lambda F(x)\hspace{1cm}\forall x\in \mathbb{R}%
_{+}^{n},\lambda \geq 1.  \label{strongcorad}
\end{equation}

The preceding definitions are generalizations of the corresponding ones for
real-valued functions. Indeed, one can associate to $f:\mathbb{R}%
_{+}^{n}\rightarrow \mathbb{R}_{+}\cup \left\{ +\infty \right\} $ the
set-valued mapping $\left[ f\right] :\mathbb{R}_{+}^{n}\rightrightarrows 
\mathbb{R}_{+}$ defined by%
\begin{equation*}
\left[ f\right] \left( x\right) :=\left[ 0,f\left( x\right) \right] ,
\end{equation*}%
with the convention $\left[ 0,+\infty \right] :=\mathbb{R}_{+}.$ Then, it is
easy to see that $f$ is increasing if, and only if, $\left[ f\right] $ is
increasing, and $f$ is co-radiant if, and only if, $\left[ f\right] $ is
co-radiant. Notice that $gr\left( \left[ f\right] \right) =hypo(f).$

For a function $f:\mathbb{R}_{+}^{n}\rightarrow \mathbb{R}_{+}\cup \left\{
+\infty \right\} ,$ one can easily check that the function $\widehat{f}:%
\mathbb{R}_{+}^{n}\rightarrow \mathbb{R}_{+}\cup \left\{ +\infty \right\} $
defined by $\widehat{f}\left( x\right) :=\sup_{\lambda \geq 1}\frac{f\left(
\lambda x\right) }{\lambda }$ is the smallest co-radiant majorant of $f.$
Similarly, for set-valued mappings we have the following proposition.

\begin{proposition}
For a mapping $F:\mathbb{R}_{+}^{n}\rightrightarrows \mathbb{R}_{+}^{m}$
with nonempty graph, the set-valued mapping $\widehat{F}:\mathbb{R}%
_{+}^{n}\rightrightarrows \mathbb{R}_{+}^{m}$ defined by%
\begin{equation}
\widehat{F}(x):=\displaystyle\bigcup_{\lambda \geq 1}\dfrac{F(\lambda x)}{%
\lambda }  \label{def H}
\end{equation}%
is the pointwise smallest (in the sense of inclusion) co-radiant majorant of 
$F$. If $F$ takes only normal values, then $\widehat{F}$ takes only normal
values, too.
\end{proposition}

\begin{proof}
Let $x\in \mathbb{R}_{+}^{n}.$ Taking $\lambda :=1$ in (\ref{def H}), we see
that $F\left( x\right) \subset \widehat{F}\left( x\right) ;$ thus $\widehat{F%
}$ is a majorant of $F.$ The mapping $\widehat{F}$ is also co-radiant, since
for $\beta \geq 1$ we have%
\begin{equation*}
\widehat{F}\left( \beta x\right) =\displaystyle\bigcup_{\lambda \geq 1}%
\dfrac{F(\lambda \beta x)}{\lambda }=\beta \displaystyle\bigcup_{\lambda
\geq 1}\dfrac{F(\lambda \beta x)}{\lambda \beta }=\beta \displaystyle%
\bigcup_{\mu \geq \beta }\dfrac{F(\mu x)}{\mu }\subset \beta \displaystyle%
\bigcup_{\mu \geq 1}\dfrac{F(\mu x)}{\mu }=\beta \widehat{F}\left( x\right) .
\end{equation*}%
If $G$ is a co-radiant majorant of $F$, then for every $\lambda \geq 1$ we
have%
\begin{equation*}
\widehat{F}(x)=\displaystyle\bigcup_{\lambda \geq 1}\dfrac{F(\lambda x)}{%
\lambda }\subset \displaystyle\bigcup_{\lambda \geq 1}\dfrac{G(\lambda x)}{%
\lambda }\subset \displaystyle\bigcup_{\lambda \geq 1}\dfrac{\lambda G(x)}{%
\lambda }=G(x);
\end{equation*}%
this proves that $\widehat{F}$ is the smallest co-radiant majorant of $F.$
Finally, if $F$ takes only normal values then, since each set $\dfrac{%
F(\lambda x)}{\lambda }$ $\left( \lambda \geq 1\right) $ is normal, their
union $\widehat{F}(x)$ is normal, too; so $\widehat{F}$ takes only normal
values.
\end{proof}

\bigskip

For $f:\mathbb{R}_{+}^{n}\rightarrow \mathbb{R}_{+}\cup \left\{ +\infty
\right\} ,$ a straightforward computation shows that, for every $x\in 
\mathbb{R}_{+}^{n},$ one has%
\begin{equation*}
\widehat{\left[ f\right] }\left( x\right) =\left\{ 
\begin{array}{c}
\left[ \widehat{f}\right] \left( x\right) \text{\qquad if either }%
\sup_{\lambda \geq 1}\frac{f\left( \lambda x\right) }{\lambda }\text{ is
attained or is }+\infty , \\ 
\left[ 0,\widehat{f}\left( x\right) \right) \text{\qquad otherwise.\qquad
\qquad \qquad \qquad \qquad \qquad \qquad\ }%
\end{array}%
\right.
\end{equation*}%
Hence%
\begin{equation*}
\left[ \widehat{f}\right] \left( x\right) =cl\widehat{\left[ f\right] }%
\left( x\right) .
\end{equation*}

The following proposition generalizes the above mentioned equivalence of
co-radiantness of a function $f:\mathbb{R}_{+}^{n}\rightarrow \mathbb{R}%
_{+}\cup \left\{ +\infty \right\} $ and radiantness of its hypograph.

\begin{proposition}
\label{caracterizacionRadiante}Let $F:\mathbb{R}_{+}^{n}\rightrightarrows 
\mathbb{R}_{+}^{m}$. Then $F$ is co-radiant if, and only if, the set $gr(F)$
is radiant$.$
\end{proposition}

\begin{proof}
Assume first that $F$ is co-radiant. Let $(x,y)\in gr(F)$ and $t\in (0,1],$
then $ty\in tF(x)\subset F(tx)$, therefore $t(x,y)=\left( tx,ty\right) \in
gr(F)$. Conversely, assume that $gr(F)$ is radiant and let $y\in F(x)$; then 
$(x,y)\in gr(F)$ and, for $t\in (0,1],$ we have $(tx,ty)=t(x,y)\in gr(F)$
and thus $ty\in F(tx)$. This proves that $tF(x)\subset F(tx).$
\end{proof}

\bigskip

We are going to consider increasing co-radiant (briefly, ICR) set-valued
mappings $F:\mathbb{R}_{+}^{n}\rightrightarrows \mathbb{R}_{+}^{m}.$ We
first present a simple result on the domains of such functions.

\begin{proposition}
Let $F:\mathbb{R}_{+}^{n}\rightrightarrows \mathbb{R}_{+}^{m}$ be ICR. If $%
\mathbb{R}_{++}^{n}\cap dom(F)\neq \emptyset ,$ then $\mathbb{R}%
_{++}^{n}\subset dom(F).$
\end{proposition}

\begin{proof}
For $x,y\in \mathbb{R}_{++}^{n},$ setting $t:=\max \left\{ \max_{i}\frac{%
x_{i}}{y_{i}},1\right\} ,$ we have $x\leq ty;$ hence $F\left( x\right)
\subset F\left( ty\right) \subset tF\left( y\right) .$ Therefore, if $x\in
dom(F),$ then $y\in dom(F),$ too.
\end{proof}

\bigskip

We are going to use the following notion of Lipschitz set-valued mapping.

\begin{definition}
(see \cite{Aubin}). Let $F:\mathbb{R}_{+}^{n}\rightrightarrows \mathbb{R}%
_{+}^{m}$ and $K\subset dom(F)$. One says that $F$ is Lipschitz on $K$ if
there exists $M>0$ such that%
\begin{equation}
F(x)\subset F(y)+M\Vert x-y\Vert _{\infty }\mathbf{B}_{\infty }\hspace{1cm}%
\forall x,y\in K.  \label{Lipschitz}
\end{equation}
\end{definition}

Notice that, in the preceding definition, one can replace $\Vert \cdot \Vert
_{\infty }$ and $\mathbf{B}_{\infty }$ with any other norm and its
corresponding unit ball, respectively, since all norms in $\mathbb{R}^{n}$
are equivalent.

It is easy to check that a function $f:\mathbb{R}_{+}^{n}\rightarrow \mathbb{%
R}_{+}\cup \left\{ +\infty \right\} $ is Lipschitz on a set $K\subset 
\mathbb{R}_{+}^{n}$ if and only if its set-valued version $\left[ f\right] $
is Lipschitz on $K$.

The following proposition generalizes \cite[Remark 4.4]{Sharikov}.

\begin{proposition}
Let $F:\mathbb{R}_{+}^{n}\rightrightarrows \mathbb{R}_{+}^{m}$ be an ICR
mapping which takes only normal values. If $F\left( \overline{x}\right) $ is
bounded for some $\overline{x}:=(\overline{x}_{1},\cdots ,\overline{x}%
_{n})\in \mathbb{R}_{++}^{n},$ then $F$ is Lipschitz on every compact set $%
K\subset \mathbb{R}_{++}^{n}$.
\end{proposition}

\begin{proof}
We will first consider the case when $F$ is positively homogeneous on $%
\mathbb{R}_{++}^{n}$. Let $x\in \mathbb{R}_{++}^{n}$ and $y=(y_{1},\cdots
,y_{n})\in K\mathbb{.}$ By the boundedness of $F(\overline{x})$, there
exists $L>0$ tal que $F(\overline{x})\subset L\mathbf{B}_{\infty }$. Since$%
\,x\leq \left( \max_{i}\frac{x_{i}}{y_{i}}\right) y$, $y\leq \left( \max_{i}%
\frac{y_{i}}{\overline{x}_{i}}\right) \overline{x}$ and, for every $%
i=1,...,n,$%
\begin{equation*}
\displaystyle\frac{x_{i}}{y_{i}}\leq 1+\dfrac{\max_{i}|x_{i}-y_{i}|}{y_{i}}%
\leq 1+\dfrac{\Vert x-y\Vert _{\infty }}{\min_{j}y_{j}},
\end{equation*}%
we have%
\begin{eqnarray*}
F(x) &\subset &\displaystyle\left( \max_{i}\dfrac{x_{i}}{y_{i}}\right)
F(y)\subset \left( 1+\dfrac{\Vert x-y\Vert _{\infty }}{\min_{j}y_{j}}\right)
F(y)\subset F\left( y\right) +\dfrac{\Vert x-y\Vert _{\infty }}{\min_{j}y_{j}%
}F(y) \\
&\subset &F\left( y\right) +\dfrac{\Vert x-y\Vert _{\infty }}{\min_{j}y_{j}}%
(\max_{j}\frac{y_{j}}{\overline{x}_{j}})F(\overline{x})\subset F\left(
y\right) +\dfrac{\max_{j}\frac{y_{j}}{\overline{x}_{j}}}{\min_{j}y_{j}}\Vert
x-y\Vert _{\infty }L\mathbf{B}_{\infty };
\end{eqnarray*}%
thus (\ref{Lipschitz}) holds with $M:=L\max_{y\in K}\dfrac{\max_{j}\frac{%
y_{j}}{\overline{x}_{j}}}{\min_{j}y_{j}}.$

In the general case, we extend $F$ to a mapping $\widetilde{F}:\mathbb{R}%
_{+}^{n+1}\rightrightarrows \mathbb{R}_{+}^{m}$ defined by%
\begin{equation*}
\widetilde{F}(x,\lambda ):=\left\{ 
\begin{array}{cc}
\lambda F(\dfrac{x}{\lambda }), & \text{if }\lambda >0, \\ 
\emptyset , & \text{if }\lambda =0.%
\end{array}%
\right.
\end{equation*}%
The mapping $\widetilde{F}$ is positively homogeneous on $\mathbb{R}%
_{++}^{n} $, increasing, and takes only normal values. Hence, as $\widetilde{%
F}(\overline{x},1)=F\left( \overline{x}\right) $ is nonempty and bounded,
the result follows by applying the first part of the proof to $\widetilde{F}$
and the compact set $K\times \left\{ 1\right\} .$\newline
\end{proof}

\bigskip

The folllowing (semi)continuity notions are standard in set-valued analysis.

\begin{definition}
(see \cite{Aubin}). (i) One says that $F:\mathbb{R}_{+}^{n}\rightrightarrows 
\mathbb{R}_{+}^{m}$ is upper semicontinuous (u.s.c.) at $x\in dom(F)$ if for
every open set $V$ in $\mathbb{R}_{+}^{m}$ such that $F(x)\subset V$, there
exists an open neighborhood $U$ of $x$ such that $F(U)\subset V$.

(ii) $F$ is lower semicontinuous (l.s.c.) at $x\in dom(F)$ if for every open
set $W$ in $\mathbb{R}_{+}^{m}$ such that $W\cap F(x)\neq \emptyset $, there
exists an open neighborhood $V$ of $x$ such that $W\cap F(x^{\prime })\neq
\emptyset $ for every $x^{\prime }\in V.$

(iii) $F$ is continuous at $x\in dom(F)$ if it is both u.s.c. and l.s.c. at $%
x$.
\end{definition}

It is worth observing that a function $f:\mathbb{R}_{+}^{n}\rightarrow 
\mathbb{R}_{+}\cup \left\{ +\infty \right\} $ is upper (lower)
semicontinuous at a point $x\in \mathbb{R}_{+}^{n}$ if, and only if, its
set-valued counterpart $\left[ f\right] $ is u.s.c. (resp., l.s.c.) at this
point.

A useful sequential characterization of upper semicontinuity is stated in
the next proposition.

\begin{proposition}
\cite[Theorem 1]{Hildebrand} Let $F:\mathbb{R}_{+}^{n}\rightrightarrows 
\mathbb{R}_{+}^{m}$ take only compact values. Then it is u.s.c. at $x\in
dom(F)$ if, and only if, for every sequence $\{x_{n}\}$ in $\mathbb{R}%
_{+}^{n}$ converging to $x$ and every sequence $\left\{ y_{n}\right\} $ in $%
\mathbb{R}_{+}^{m}$ such that $y_{n}\in F(x_{n})$ for every $n$ there exists
a subsequence of $\left\{ y_{n}\right\} $ which converges to a point in $%
F(x).$
\end{proposition}

The following corollary is immediate. For $x\in X$ and $A\subset X,$ we
define $d_{\infty }(x,A):=\inf_{a\in A}\left\Vert x-a\right\Vert _{\infty }.$

\begin{corollary}
\label{char usc}Let $F:\mathbb{R}_{+}^{n}\rightrightarrows \mathbb{R}%
_{+}^{m} $ take only compact values. Then it is u.s.c. at $x\in dom(F)$ if,
and only if, for every sequence $\{x_{n}\}$ in $\mathbb{R}_{+}^{n}$
converging to $x$ and every sequence $\left\{ y_{n}\right\} $ in $\mathbb{R}%
_{+}^{m}$ such that $y_{n}\in F(x_{n})$ for every $n,$ one has $\displaystyle%
\lim_{n\rightarrow +\infty }d_{\infty }(y_{n},F(x))=0.$
\end{corollary}

\bigskip

We next present generalizations of some results obtained in \cite{Martinez}.

\begin{proposition}
If $F:\mathbb{R}_{+}^{n}\rightrightarrows \mathbb{R}_{+}^{m}$ is an ICR\
mapping and takes only normal values and there exists $\overline{x}\in 
\mathbb{R}_{++}^{n}$ such that $F(\overline{x})$ is bounded, then $F(x)$ is
bounded for every $x\in \mathbb{R}_{+}^{n}$. Furthermore, if $F(\overline{x}%
)=\{0\}$ ($F(\overline{x})=\emptyset $) then $F(x)=\{0\}$ ($F(x)=\emptyset ,$
respectively) for every $x\in \mathbb{R}_{+}^{n}$.
\end{proposition}

\begin{proof}
For every $x\in \mathbb{R}_{+}^{n}$ existe $\lambda \geq 1$ such that $x\leq
\lambda \overline{x}$, hence $F(x)\subset F(\lambda \overline{x})\subset
\lambda F(\overline{x})$.
\end{proof}

\begin{proposition}
If $F:\mathbb{R}_{+}^{n}\rightrightarrows \mathbb{R}_{+}^{m}$ is an ICR\
mapping and takes only normal values, then

\begin{itemize}
\item[(i)] $F$ is l.s.c. at every $x\in dom(F)\cap \mathbb{R}_{++}^{n}.$

\item[(ii)] If $\ F$ takes only compact values, then it is continuous at
every $x\in dom(F)\cap \mathbb{R}_{++}^{n}.$
\end{itemize}
\end{proposition}

\begin{proof}
(i) Let $W$ be an open set in $\mathbb{R}_{+}^{m}$ such that $W\cap F(x)\neq
\emptyset .$ Take $y\in W\cap F(x).$ For small enough $\delta >0$ and $%
U:=\{ty:t\in \lbrack 1-\delta ,1]\},$ we have $U\subset W.$ Set $V:=\left(
1-\delta \right) x+\mathbb{R}_{++}^{n}$, then $V$ is open and $x=\left(
1-\delta \right) x+\delta x\in \left( 1-\delta \right) x+\mathbb{R}%
_{++}^{n}=V;$ moreover, for every $x^{\prime }\in V,$ we have $\left(
1-\delta \right) y\in U\cap \left( 1-\delta \right) F(x)\subset W\cap
F((1-\delta )x)\subset W\cap F(x^{\prime }),$ which shows that $W\cap
F(x^{\prime })\neq \emptyset .$

(ii) Let $\{x^{k}\}$ be a sequence in $\mathbb{R}_{+}^{n}$ converging to $x$
and $\left\{ y^{k}\right\} $ be a sequence in $\mathbb{R}_{+}^{m}$ such that 
$y^{k}\in F(x^{k})$ for every $k.$ Fix $\epsilon >0,$ and let $M>\epsilon $
be the radius of an open ball centered at the origin which contains $F\left(
x\right) .$ Since $\{x^{k}\}$ converges to $x,$ for sufficiently large $k$
we have $x^{k}\leq \frac{M}{M-\epsilon }x;$ hence, as $F$ is increasing and
co-radiant, we deduce that $F(x^{k})\subset \frac{M}{M-\epsilon }F(x)$,
which implies that $y^{k}\in \frac{M}{M-\epsilon }F(x)$. Therefore 
\begin{equation*}
d_{\infty }(y^{k},F(x))\leq d(y^{k},\frac{M-\epsilon }{M}y^{k})=\frac{%
\epsilon }{M}\Vert y^{k}\Vert <\epsilon ,
\end{equation*}%
which proves that $\displaystyle{\lim_{k\rightarrow +\infty }}d_{\infty
}(y^{k},F(x))=0.$ Hence, by Corollary \ref{char usc}, the mapping $F$ is
u.s.c. at $x.$ Continuity follows from statement (i).
\end{proof}

\bigskip

The proofs of the following propositions are immediate.

\begin{proposition}
\label{utilparalaPropsiguiente}If $F:\mathbb{R}_{+}^{n}\rightrightarrows 
\mathbb{R}_{+}^{m}$ is co-radiant, then the mappings $\overline{F},\mathbf{B}%
_{\epsilon }(F):\mathbb{R}_{+}^{n}\rightrightarrows \mathbb{R}_{+}^{m}$
defined by%
\begin{equation*}
\overline{F}(x):=\overline{F(x)}\hspace{0.25cm}\text{and }\hspace{0.25cm}%
\mathbf{B}_{\epsilon }(F)(x):=\{y\in \mathbb{R}^{m}\,:\,d_{\infty
}(y,F(x))\leq \epsilon \}
\end{equation*}%
are co-radiant.
\end{proposition}

As is well known, the class of co-radiant functions is closed both under
pointwise infimum and pointwise supremum. The following proposition provides
set-valued generalizations of these facts.

\begin{proposition}
\label{utilparalaPropsiguiente1}If $\{F^{i}\}_{i\in \mathcal{I}}$ is a
family of co-radiant mappings, then the mappings $\bigcap_{i\in \mathcal{I}%
}F^{i}$ and $\bigcup_{i\in \mathcal{I}}F^{i}$ defined by 
\begin{equation*}
\displaystyle\left( \bigcap_{i\in \mathcal{I}}F^{i}\right) (x):=\displaystyle%
\bigcap_{i\in \mathcal{I}}F^{i}(x)\hspace{0.25cm}\text{and}\hspace{0.25cm}%
\displaystyle\left( \bigcup_{i\in \mathcal{I}}F^{i}\right) (x):=\displaystyle%
\bigcup_{i\in \mathcal{I}}F^{i}(x)
\end{equation*}%
are co-radiant.
\end{proposition}

\bigskip

We are going to deal with pointwise Painlev\'{e}-Kuratowski limits of
sequences of set-valued mappings. We first recall the convergence notions in
the Painlev\'{e}-Kuratowski sense for sequences of subsets.

\begin{definition}
Let $\{C_{k}\}$ be a sequence of subsets of $\mathbb{R}^{n}$. The sets%
\begin{equation*}
{\mathrm{Limsup}}_{k}C_{k}:=\left\{ x\in \mathbb{R}^{n}\text{ }:\lim
\inf_{k\rightarrow +\infty }d_{\infty }(x,C_{k})=0\right\}
\end{equation*}%
and%
\begin{equation*}
{\mathrm{Liminf}}_{k}C_{k}:=\left\{ x\in \mathbb{R}^{n}\text{ }%
:\lim_{k\rightarrow +\infty }d_{\infty }(x,C_{k})=0\right\}
\end{equation*}%
are called the $upper$ $limit$ and the $lower$ $limit$ of $\{C_{k}\}$,
respectively$.$
\end{definition}

\bigskip

It is easy to see that these sets do not change if one replaces the sets $%
C_{k}$ with their closures and that ${\mathrm{Liminf}}_{k}C_{k}\subset $ ${%
\mathrm{Limsup}}_{k}C_{k}.$

The following representations of ${\mathrm{Liminf}}$ and ${\mathrm{Limsup}}$
are very useful.

\begin{proposition}
\label{Important}\cite[p. 21]{Aubin}\cite[Exercise 4.2(b)]{Rockafellar-Wets}
For a sequence $\{C_{k}\}$ of subsets of $\mathbb{R}^{n},$ one has:
\end{proposition}

\begin{itemize}
\item[(i)] ${\mathrm{Limsup}}_{k}C_{k}=\mathbb{\dbigcap }_{l\in \mathbb{N}}%
\overline{\mathbb{\dbigcup }_{k\geq l}C_{k}}.$

\item[(ii)] ${\mathrm{Liminf}}_{k}C_{k}=\mathbb{\dbigcap }_{I\in \aleph }%
\overline{\mathbb{\dbigcup }_{i\in I}C_{i}}.$
\end{itemize}

\noindent Here $\aleph $ denotes the set of infinite subsets of $\mathbb{N}.$

\bigskip

From Proposition \ref{Important}, it immediately follows that the sets ${%
\mathrm{Liminf}}_{k}C_{k}$ and ${\mathrm{Limsup}}_{k}C_{k}$ are closed.

\begin{definition}
\cite[Chapter 5]{Rockafellar-Wets} Let $\{F^{k}\}$ be a sequence of
set-valued mappings from ${\mathbb{R}}_{+}^{n}$ into $\mathbb{R}_{+}^{m}$.

\begin{itemize}
\item[(i)] $\mathbf{The}$ $\mathbf{pointwise\ lower\ limit}$ of $\{F^{k}\}$
is the mapping ${\mathrm{Liminf}}_{k}F^{k}:\mathbb{R}_{+}^{n}%
\rightrightarrows \mathbb{R}_{+}^{m}$ defined by 
\begin{equation*}
({\mathrm{Liminf}}_{k}F^{k})(x):={\mathrm{Liminf}}_{k}F^{k}(x)
\end{equation*}

\item[(ii)] $\mathbf{The}$ $\mathbf{pointwise}$ $\mathbf{upper}$ $\mathbf{%
limit}$ of $\{F^{k}\}$ is the mapping ${\mathrm{Limsup}}_{k}F^{k}:\mathbb{R}%
_{+}^{n}\rightrightarrows \mathbb{R}_{+}^{m}$ defined by%
\begin{equation*}
({\mathrm{Limsup}}_{k}F^{k})(x):={\mathrm{Limsup}}_{k}F^{k}(x)
\end{equation*}
\end{itemize}
\end{definition}

The following proposition collects some basic properties of pointwise lower
and upper limits of set-valued mappings.

\begin{proposition}
For a sequence $\{F^{k}\}$ of set-valued mappings from $\mathbb{R}_{+}^{n}$
into $\mathbb{R}_{+}^{m}$, one has:

\begin{itemize}
\item[(i)] ${\mathrm{Liminf}}_{k}F^{k}$\thinspace\ and\thinspace\ ${\mathrm{%
Limsup}}_{k}F^{k}$ only take closed values.

\item[(ii)] If each $F^{k}$ only takes normal values, then ${\mathrm{Liminf}}%
_{k}F^{k}$\thinspace\ and\thinspace\ ${\mathrm{Limsup}}_{k}F^{k}$ only take
normal values, too.

\item[(iii)] If each $F^{k}$ is increasing, then ${\mathrm{Liminf}}_{k}F^{k}$%
\thinspace\ and\thinspace\ ${\mathrm{Limsup}}_{k}F^{k}$ are increasing, too.

\item[(iv)] If each $F^{k}$ is co-radiant, then ${\mathrm{Liminf}}_{k}F^{k}$%
\thinspace\ and\thinspace\ ${\mathrm{Limsup}}_{k}F^{k}$ are co-radiant, too.
\end{itemize}
\end{proposition}

\begin{proof}
Statement (i) directly follows from the definitions of ${\mathrm{Liminf}}$%
\thinspace\ and\thinspace\ ${\mathrm{Limsup}}$. Statements (ii) and (iii)
follow from Proposition \ref{Important}. Finally, statement (iv) follows
from Propositions \ref{utilparalaPropsiguiente}, \ref%
{utilparalaPropsiguiente1} and \ref{Important}.
\end{proof}

\bigskip

We next give several propositions providing characterizations of some
properties of set-valued mappings in terms of their graphs.

\begin{proposition}
\label{caracterizaCreciente}If $F:\mathbb{R}_{+}^{n}\rightrightarrows 
\mathbb{R}_{+}^{m}$ takes only normal values, then 
\begin{equation*}
F\text{ is increasing}\Leftrightarrow gr(F)+\mathbb{R}_{+}^{n}\times
\{0\}=gr(F)
\end{equation*}
\end{proposition}

\begin{proof}
$\Rightarrow $ ) Obviously, $gr(F)\subset gr(F)+\mathbb{R}_{+}^{n}\times
\{0\}$. To prove the opposite inclusion, let $\left( x,y\right) \in gr(F)$
and $p\in \mathbb{R}_{+}^{n};$ then $y\in F\left( x\right) \subset F\left(
x+p\right) ,$ so we conclude that $\left( x,y\right) +\left( p,0\right) \in
gr(F).$

$\Leftarrow $ ) Let $x^{\prime }\geq x\in \mathbb{R}_{+}^{n}$ and $y\in F(x)$%
, then $(x^{\prime },y)=(x,y)+(x^{\prime }-x,0)\in gr(F)+\mathbb{R}%
_{+}^{n}\times \{0\}=gr(F)$, therefore $y\in F(x^{\prime }).$
\end{proof}

\begin{proposition}
\label{caracterizaNormalvalues}Let $F:\mathbb{R}_{+}^{n}\rightrightarrows 
\mathbb{R}_{+}^{m}$. Then%
\begin{equation*}
F\text{ takes only normal values}\Leftrightarrow \left( gr(F)+\{0\}\times (-%
\mathbb{R}_{+}^{m})\right) \cap \mathbb{R}_{+}^{n}\times \mathbb{R}%
_{+}^{m}=gr(F).
\end{equation*}
\end{proposition}

\begin{proof}
$\Rightarrow $ ) Obviously, $gr(F)\subset \left( gr(F)+\{0\}\times (-\mathbb{%
R}_{+}^{m})\right) \cap \mathbb{R}_{+}^{n}\times \mathbb{R}_{+}^{m}.$ To
prove the opposite inclusion, let $(x,y)\in gr(F)$ and $p\in \mathbb{R}%
_{+}^{m}$ be such that $y-p\in \mathbb{R}_{+}^{m}$, then, by the normality
of $F(x)$, one has $(x,y-p)\in gr(F)$.

$\Leftarrow $ ) If $x\in \mathbb{R}_{+}^{n}$ and $0\leq y^{\prime }\leq y\in
F(x)$ we have $(x,y^{\prime })=(x,y)+(0,y^{\prime }-y)\in \left(
gr(F)+\{0\}\times (-\mathbb{R}_{+}^{m})\right) \cap \mathbb{R}_{+}^{n}\times 
\mathbb{R}_{+}^{m}=gr(F).$ Hence $y^{\prime }\in F\left( x\right) ,$ which
proves that $F\left( x\right) $ is normal.
\end{proof}

\begin{proposition}
\label{IncNV}Let $F:\mathbb{R}_{+}^{n}\rightrightarrows \mathbb{R}_{+}^{m}$.
Then $F$ is increasing and takes only normal values if, and only if, for
every $\left( x,y\right) \in \mathbb{R}_{+}^{n}\times \mathbb{R}%
_{+}^{m}\setminus gr(F)$ one has%
\begin{equation}
gr(F)\cap (\left( x,y\right) +(-\mathbb{R}_{+}^{n})\times \mathbb{R}%
_{+}^{m})=\emptyset .  \label{disj part}
\end{equation}
\end{proposition}

\begin{proof}
If $F$ is increasing and takes only normal values, the existence of $\left(
p_{1},p_{2}\right) \in \mathbb{R}_{+}^{n}\times \mathbb{R}_{+}^{m}$
satisfying $\left( x,y\right) +(-p_{1},p_{2})\in gr(F),$ that is, $%
y+p_{2}\in F\left( x-p_{1}\right) ,$ by the increasingness of $F$ would
imply $y+p_{2}\in F\left( x\right) ;$ hence, by the normality of $F\left(
x\right) ,$ we would have $y\in F\left( x\right) ,$ that is, $\left(
x,y\right) \in gr(F),$ a contradiction. To prove the converse implication,
let $x,x^{\prime }\in \mathbb{R}_{+}^{n}$ be such that $x\leq x^{\prime },$
and let $y\in \mathbb{R}_{+}^{m}\setminus F\left( x^{\prime }\right) .$
Since $\left( x,y\right) =\left( x^{\prime },y\right) +\left( x-x^{\prime
},0\right) \in \left( x,y\right) +(-\mathbb{R}_{+}^{n})\times \mathbb{R}%
_{+}^{m},$ by (\ref{disj part}) we have $\left( x,y\right) \notin gr(F),$
that is, $y\notin F\left( x\right) ,$ which shows that $F$ is increasing. To
prove that $F$ takes only normal values, let $x\in \mathbb{R}_{+}^{n}$ and $%
y,y^{\prime }\in \mathbb{R}_{+}^{m}$ be such that $y\leq y^{\prime }$ and $%
y\in \mathbb{R}_{+}^{m}\setminus F\left( x\right) .$ Since $\left(
x,y^{\prime }\right) =\left( x,y\right) +\left( 0,y^{\prime }-y\right) \in
\left( x,y\right) +(-\mathbb{R}_{+}^{n})\times \mathbb{R}_{+}^{m}),$ by (\ref%
{disj part}) we have $\left( x,y^{\prime }\right) \notin gr(F),$ that is, $%
y^{\prime }\notin F\left( x\right) ,$ which shows that $F\left( x\right) $
is normal.
\end{proof}

\section{Representations of ICR Mappings as Intersections of Elementary
Mappings}

In this section we will introduce two notions of elementary ICR mappings,
and we will show that they generate all the ICR\ mappings that satisfy some
suitable additional properties.

The following separation result of radiant sets by convex cones will be used
to separate graphs of co-radiant set-valued mappings.

\begin{lemma}
\label{Zaff}(see \cite[Proposition 3.4]{Zaffaroni}) A nonempty closed set $%
A\subset \mathbb{R}_{+}^{p}$ is radiant if, and only if, for every $x\in 
\mathbb{R}_{+}^{p}\setminus A$ there exists a cone $K_{A,x}\subset \mathbb{R}%
_{+}^{p}$ such that $x\in K_{A,x}$ and $A\cap (x+K_{A,x})=\emptyset .$ One
can take $K_{A,x}$ convex and such that $K_{A,x}\setminus \{0\}$ is open in $%
\mathbb{R}_{+}^{p},$ namely%
\begin{equation}
K_{A,x}:=\left\{ \dsum\limits_{i=1}^{p}\lambda _{i}\left(
x+r_{A,x}e_{i}\right) :\lambda _{i}>0\text{ for }i\in I_{+}\left( x\right) ,%
\text{ }\lambda _{i}\geq 0\text{ for }i\in I\setminus I_{+}\left( x\right)
\right\} \cup \left\{ 0\right\} ,  \label{cone}
\end{equation}%
with the $e_{i}^{\text{' }}$s denoting the unit vectors and $%
r_{A,x}:=d_{\infty }(x,A).$
\end{lemma}

\begin{proof}
Assume first that $A$ is radiant. For $x\in \mathbb{R}_{+}^{p}\setminus A,$
the cone $K_{A,x}$ is convex, $K_{A,x}\setminus \{0\}$ is open in $\mathbb{R}%
_{+}^{p}$, and $x\in K_{A,x}$ (indeed, take $\lambda _{i}:=\frac{x_{i}}{%
\dsum\limits_{j=1}^{p}x_{j}+r_{A,x}})$. Assume that $x+\dsum%
\limits_{i=1}^{p}\lambda _{i}\left( x+r_{A,x}e_{i}\right) \in A$ for some $%
\lambda _{i}>0$ ($i\in I_{+}\left( x\right) $) and $\lambda _{i}\geq 0$ ($%
i\in I\setminus I_{+}\left( x\right) $). Then, as $A$ is radiant, we have $x+%
\frac{r_{A,x}}{1+\dsum\limits_{i=1}^{p}\lambda _{i}}\dsum\limits_{i=1}^{p}%
\lambda _{i}e_{i}=\frac{1}{1+\dsum\limits_{i=1}^{p}\lambda _{i}}\left(
x+\dsum\limits_{i=1}^{p}\lambda _{i}\left( x+r_{A,x}e_{i}\right) \right) \in
A,$ which contradicts the definition of $r_{A,x}.$ This proves that $A\cap
(x+K_{A,z})=\emptyset .$

Conversely, let $a\in A\setminus \left\{ 0\right\} $ and $t\in (0,1].$ If $%
ta\notin A,$ by assumption there exists a cone $K_{A,ta}\subset \mathbb{R}%
_{+}^{p}$ such that $ta\in K_{A,ta}$ and $A\cap (ta+K_{A,ta})=\emptyset .$
However this is impossible, since $a=ta+\frac{1-t}{t}ta\in ta+K_{A,ta}.$
Hence $ta\in A,$ which proves that $A$ is radiant.
\end{proof}

\begin{proposition}
\label{PropPreparatoria} If $F:\mathbb{R}_{+}^{n}\rightrightarrows \mathbb{R}%
_{+}^{m}$ is co-radiant and $gr(F)$ is closed, then

\begin{itemize}
\item[(a)] For every $\left( x,y\right) \in \mathbb{R}_{+}^{n}\times \mathbb{%
R}_{+}^{m}\setminus gr(F)$ there exists a convex cone $K$ in $\mathbb{R}%
_{+}^{n}\times \mathbb{R}_{+}^{m}$ containing $\left( x,y\right) $ such that 
$K\setminus \{\left( 0,0\right) \}$ is open in $\mathbb{R}_{+}^{n}\times 
\mathbb{R}_{+}^{m}$ and%
\begin{equation}
gr(F)\cap (\left( x,y\right) +K\setminus \left\{ \left( 0,0\right) \right\}
)=\emptyset ,  \label{disj gen}
\end{equation}%
namely one can take $K:=K_{gr\left( F\right) ,\left( x,y\right) },$ the cone
defined in (\ref{cone}), with $p:=n+m,$ $z:=\left( x,y\right) ,$ $A:=gr(F)$
and $r_{gr(F),\left( x,y\right) }:=d_{\infty }(\left( x,y\right) ,gr(F)).$

\item[(b)] If $F$ takes only normal values, then for every $\left(
x,y\right) \in \mathbb{R}_{+}^{n}\times \mathbb{R}_{+}^{m}\setminus gr(F)$
there exists a convex cone $K$ in $\mathbb{R}_{+}^{n}\times \mathbb{R}%
_{+}^{m}$ containing $\left( x,y\right) $ such that $K\setminus \{\left(
0,0\right) \}$ is open in $\mathbb{R}_{+}^{n}\times \mathbb{R}_{+}^{m}$ and%
\begin{equation}
gr(F)\cap (\left( x,y\right) +K+\{0\}\times \mathbb{R}_{+}^{m})=\emptyset ,
\label{disj}
\end{equation}%
namely one can take $K:=K_{gr\left( F\right) ,\left( x,y\right) },$ the cone
considered in (a).

\item[(c)] If $F$ is increasing and takes only normal values, then for every 
$\left( x,y\right) \in \mathbb{R}_{+}^{n}\times \mathbb{R}_{+}^{m}\setminus
gr(F)$ there exists a convex cone $K$ in $\mathbb{R}_{+}^{n}\times \mathbb{R}%
_{+}^{m}$ containing $\left( x,y\right) $ such that $K\setminus \{\left(
0,0\right) \}$ is open in $\mathbb{R}_{+}^{n}\times \mathbb{R}_{+}^{m}$ and%
\begin{equation}
gr(F)\cap \left( (x,y)+K+(-\mathbb{R}_{+}^{n})\times \mathbb{R}%
_{+}^{m}\right) =\emptyset ,  \label{disj ICR}
\end{equation}%
namely one can take $K:=K_{gr\left( F\right) ,\left( x,y\right) },$ the cone
considered in (a).
\end{itemize}
\end{proposition}

\begin{proof}
(a) It is an immediate consequence of Lemma \ref{Zaff}.

(b) For $K:=K_{gr\left( F\right) ,\left( x,y\right) },$ equality (\ref{disj}%
) holds, since the existence of $\left( k_{1},k_{2}\right) \in K_{gr\left(
F\right) ,\left( x,y\right) }$ and $p\in \mathbb{R}_{+}^{m}$ satisfying $%
\left( x,y\right) +\left( k_{1},k_{2}\right) +(0,p)\in gr(F),$ that is, $%
y+k_{2}+p\in F\left( x+k_{1}\right) ,$ by the normality of $F\left(
x+k_{1}\right) $ would imply $y+k_{2}\in F\left( x+k_{1}\right) ,$ that is, $%
\left( x,y\right) +\left( k_{1},k_{2}\right) \in gr(F),$ a contradiction
with (\ref{disj gen}).

(c) In view of Proposition \ref{IncNV}, the convex cone $K:=K_{gr\left(
F\right) ,\left( x,y\right) }$ satisfies (\ref{disj ICR}).
\end{proof}

\bigskip

Motivated by Proposition \ref{PropPreparatoria}(c), given a set $A\subset 
\mathbb{R}_{+}^{n}\times \mathbb{R}_{+}^{m}$ and a point $\left( x,y\right)
\in \mathbb{R}_{+}^{n}\times \mathbb{R}_{+}^{m}\setminus A$ we introduce the
mapping $E_{A,\left( x,y\right) }:\mathbb{R}_{+}^{n}\rightrightarrows 
\mathbb{R}_{+}^{m}$ defined by%
\begin{equation*}
gr(E_{A,\left( x,y\right) }):=\mathbb{R}_{+}^{n}\times \mathbb{R}%
_{+}^{m}\setminus \left( (x,y)+K_{A,\left( x,y\right) }\setminus \left\{
\left( 0,0\right) \right\} +(-\mathbb{R}_{+}^{n})\times \mathbb{R}%
_{+}^{m}\right) ,
\end{equation*}%
with $K_{A,\left( x,y\right) }$ being the cone defined in Proposition \ref%
{PropPreparatoria}(a).

\begin{proposition}
\label{SobreLasFuv}For $A\subset \mathbb{R}_{+}^{n}\times \mathbb{R}_{+}^{m}$
and $\left( x,y\right) \in \mathbb{R}_{+}^{n}\times \mathbb{R}%
_{+}^{m}\setminus A,$ one has:

\begin{itemize}
\item[(i)] $gr(E_{A,\left( x,y\right) })$ is closed.

\item[(ii)] $dom(E_{A,\left( x,y\right) })=\mathbb{R}_{+}^{n}$ if, and only
if, $y\neq 0$.

\item[(iii)] $E_{A,\left( x,y\right) }$ is ICR and takes only normal values.
\end{itemize}
\end{proposition}

\begin{proof}
(i) It is a consequence of the fact that $K_{A,\left( x,y\right) }\setminus
\left\{ \left( 0,0\right) \right\} $ is open in $\mathbb{R}_{+}^{n}\times 
\mathbb{R}_{+}^{m}.$

(ii) If $y\neq 0$ and $u\in \mathbb{R}_{+}^{n},$ then $\left( u,\frac{1}{2}%
y\right) \notin (x,y)+K_{A,\left( x,y\right) }\setminus \left\{ \left(
0,0\right) \right\} +(-\mathbb{R}_{+}^{n})\times \mathbb{R}_{+}^{m},$ since $%
K_{A,\left( x,y\right) }\subset \mathbb{R}_{+}^{n}\times \mathbb{R}_{+}^{m}.$
Therefore $\frac{1}{2}y\in E_{A,\left( x,y\right) }\left( u\right) ,$ and
thus $u\in dom(E_{A,\left( x,y\right) }).$ This proves that $dom(E_{A,\left(
x,y\right) })=\mathbb{R}_{+}^{n}.$ Conversely, assume that $y=0,$ and let $%
v\in \mathbb{R}_{+}^{m}$. Since%
\begin{equation*}
\left( x,v\right) =\left( x,0\right) +\left( x,0\right) +\left( -x,v\right)
\in (x,0)+K_{A,\left( x,y\right) }\setminus \left\{ \left( 0,0\right)
\right\} +(-\mathbb{R}_{+}^{n})\times \mathbb{R}_{+}^{m},
\end{equation*}%
it follows that $\left( x,v\right) \notin gr(E_{A,\left( x,y\right) }),$
which shows that $x\notin dom(E_{A,\left( x,y\right) }).$ Thus $%
dom(E_{A,\left( x,y\right) })\neq \mathbb{R}_{+}^{n}.$ This proves the "only
if" statement.

(iii) Let $\left( u,v\right) \in \mathbb{R}_{+}^{n}\times \mathbb{R}%
_{+}^{m}\setminus gr(E_{A,\left( x,y\right) })$ and $\left( u^{\prime
},v^{\prime }\right) \in \left( u,v\right) +(-\mathbb{R}_{+}^{n})\times 
\mathbb{R}_{+}^{m}.$ Since%
\begin{eqnarray*}
\left( u^{\prime },v^{\prime }\right) &=&\left( u,v\right) +\left( u^{\prime
}-u,v^{\prime }-v\right) \\
&\in &(x,y)+K_{A,\left( x,y\right) }\setminus \left\{ \left( 0,0\right)
\right\} +(-\mathbb{R}_{+}^{n})\times \mathbb{R}_{+}^{m}+(-\mathbb{R}%
_{+}^{n})\times \mathbb{R}_{+}^{m} \\
&=&(x,y)+K_{A,\left( x,y\right) }\setminus \left\{ \left( 0,0\right)
\right\} +(-\mathbb{R}_{+}^{n})\times \mathbb{R}_{+}^{m},
\end{eqnarray*}%
we have $\left( u^{\prime },v^{\prime }\right) \notin gr(E_{A,\left(
x,y\right) }),$ which proves that%
\begin{equation*}
gr(E_{A,\left( x,y\right) })\cap \left( \left( u,v\right) +(-\mathbb{R}%
_{+}^{n})\times \mathbb{R}_{+}^{m}\right) =\emptyset .
\end{equation*}%
Hence, by Proposition \ref{IncNV}, the mapping $E_{A,\left( x,y\right) }$ is
increasing and takes only normal values.

Let $\left( u,v\right) \in (x,y)+K_{A,\left( x,y\right) }\setminus \left\{
\left( 0,0\right) \right\} +(-\mathbb{R}_{+}^{n})\times \mathbb{R}_{+}^{m}$
and $t\in \left( 0,1\right] .$ Since $K_{A,\left( x,y\right) }$ is a cone
and contains the point $\left( x,y\right) ,$ we have%
\begin{eqnarray*}
t\left( u,v\right) &\in &t(x,y)+K_{A,\left( x,y\right) }\setminus \left\{
\left( 0,0\right) \right\} +(-\mathbb{R}_{+}^{n})\times \mathbb{R}_{+}^{m} \\
&=&(x,y)+\left( t-1\right) (x,y)+K_{A,\left( x,y\right) }\setminus \left\{
\left( 0,0\right) \right\} +(-\mathbb{R}_{+}^{n})\times \mathbb{R}_{+}^{m} \\
&\subset &(x,y)+K_{A,\left( x,y\right) }\setminus \left\{ \left( 0,0\right)
\right\} +(-\mathbb{R}_{+}^{n})\times \mathbb{R}_{+}^{m},
\end{eqnarray*}%
which proves that $(x,y)+K_{A,\left( x,y\right) }\setminus \left\{ \left(
0,0\right) \right\} +(-\mathbb{R}_{+}^{n})\times \mathbb{R}_{+}^{m}$ is
co-radiant. Hence $gr(E_{A,\left( x,y\right) })$ is radiant and therefore,
by Proposition \ref{caracterizacionRadiante}, the mapping $E_{A,\left(
x,y\right) }$ is co-radiant.
\end{proof}

\bigskip

We will denote by $\mathcal{A}$ the class of nonempty closed radiant sets $%
A\subset \mathbb{R}_{+}^{n}\times \mathbb{R}_{+}^{m}$ such that $A+\mathbb{R}%
_{+}^{n}\times \{0\}=A$ and $\left( A+\{0\}\times (-\mathbb{R}%
_{+}^{m})\right) \cap \mathbb{R}_{+}^{n}\times \mathbb{R}_{+}^{m}=A,$ and we
set $\mathcal{E}:=\left\{ E_{A,\left( x,y\right) }:A\in \mathcal{A},\text{ }%
\left( x,y\right) \in \mathbb{R}_{+}^{n}\times \mathbb{R}_{+}^{m}\setminus
A\right\} .$

The follolwing theorem is one of the main results of this paper.

\begin{theorem}
\label{Intersection}A mapping $F:\mathbb{R}_{+}^{n}\rightrightarrows \mathbb{%
R}_{+}^{m}$ is ICR, takes only normal values and has a closed graph if, and
only if, there exists a set $A\in \mathcal{A}$ such that 
\begin{equation*}
F=\bigcap_{\left( x,y\right) \in \mathbb{R}_{+}^{n}\times \mathbb{R}%
_{+}^{m}\setminus A}E_{A,\left( x,y\right) },
\end{equation*}%
namely one can take%
\begin{equation}
A:=gr(F).  \label{Zeta}
\end{equation}
\end{theorem}

\begin{proof}
The "if" statement is an immediate consequence of Proposition \ref%
{SobreLasFuv} and the fact that the properties of being increasing,
co-radiant, taking only normal values and having a closed graph are
preserved by intersections. To prove the converse, we will actually prove
the equivalent equality%
\begin{equation}
gr(F)=\bigcap_{\left( x,y\right) \in \mathbb{R}_{+}^{n}\times \mathbb{R}%
_{+}^{m}\setminus A}gr\left( E_{A,\left( x,y\right) }\right) .
\label{PrimeraRepresentacion}
\end{equation}%
Assume that $F$ is ICR, takes only normal values and has a closed graph, and
define $A$ by (\ref{Zeta}). By propositions \ref{caracterizaCreciente}, \ref%
{caracterizacionRadiante} and \ref{caracterizaNormalvalues}, we have $A\in 
\mathcal{A}.$ One clearly has the inclusion $\subset $ in (\ref%
{PrimeraRepresentacion}). To prove the opposite inclusion, let $\left(
x,y\right) \in \mathbb{R}_{+}^{n}\times \mathbb{R}_{+}^{m}\setminus gr(F),$
and take $\alpha \in \left( 0,1\right) $ such that $\alpha (x,y)\in \mathbb{R%
}_{+}^{n}\times \mathbb{R}_{+}^{m}\setminus A.$ Then, as $\left( x,y\right)
=\alpha (x,y)+\frac{1-\alpha }{\alpha }\alpha (x,y)+\left( 0,0\right) \in
\alpha (x,y)+K_{A,\alpha \left( x,y\right) }\setminus \left\{ \left(
0,0\right) \right\} +(-\mathbb{R}_{+}^{n})\times \mathbb{R}_{+}^{m},$ we
have $\left( x,y\right) \notin gr(E_{A,\alpha (x,y)}).$ This proves the
inclusion $\supset $ in (\ref{PrimeraRepresentacion}) and hence the equality.
\end{proof}

\bigskip

Given $F:\mathbb{R}_{+}^{n}\rightrightarrows \mathbb{R}_{+}^{m}$, for each $%
\ell \in \mathbb{R}_{+}^{m}\setminus \{0\}$ we define $\psi _{\ell }:\mathbb{%
R}_{+}^{n}\rightarrow \mathbb{R}_{+}\cup \left\{ +\infty \right\} $ by 
\begin{equation*}
\psi _{\ell }(x):=\sigma _{F(x)}(\ell ).
\end{equation*}%
We will need the following version of the Maximum Theorem.

\begin{lemma}
\label{MaxTh} \cite[Theorem 1.4.16]{Aubin} Let $C:\mathbb{R}%
_{+}^{n}\rightrightarrows \mathbb{R}_{+}^{m}$ be u.s.c. and such that it
takes only nonempty compact values, and $f:gr(C)\rightarrow \mathbb{R}$ be
upper semicontinuous.Then the function $f^{\ast }:\mathbb{R}%
_{+}^{n}\rightarrow \mathbb{R}$ defined by $f^{\ast }(x):=\displaystyle%
\max_{y\in C(x)}f(x,y)$ is upper semicontinuous.
\end{lemma}

\begin{proposition}
\label{ParaLasiguiente}Let $F:\mathbb{R}_{+}^{n}\rightrightarrows \mathbb{R}%
_{+}^{m}$ and $\ell \in \mathbb{R}_{+}^{m}\setminus \{0\}.$

\begin{itemize}
\item[(i)] If $F$ is co-radiant, then $\psi _{\ell }$ is co-radiant. The
converse holds true if $F$ has a nonempty graph and takes only closed normal
values.

\item[(ii)] If $F$ is increasing, then $\psi _{\ell }$ is increasing. The
converse holds true if $F$ takes only closed normal values.

\item[(iii)] If $F$ is u.s.c. and takes only nonempty compact values, then $%
\psi _{\ell }$ is finite-valued and upper semicontinuous.
\end{itemize}
\end{proposition}

\begin{proof}
(i) For $x\in \mathbb{R}_{+}^{n}$ and $t\in \left( 0,1\right] $, one has $%
F(tx)\supset tF(x)$ and therefore%
\begin{equation*}
\psi _{\ell }(tx)=\sigma _{F(tx)}(\ell )\geq \sigma _{tF(x)}(\ell )=t\sigma
_{F(x)}(\ell )=t\psi _{\ell }(x).
\end{equation*}%
Conversely, assume that $F$ has a nonempty graph and takes only closed
normal values and $\psi _{\ell }$ is co-radiant, and let $x\in \mathbb{R}%
_{+}^{n}$ and $\lambda \geq 1.$ Using Proposition \ref{properties sigma}%
(vi), we obtain%
\begin{equation*}
\sigma _{F\left( \lambda x\right) }\left( l\right) =\psi _{\ell }(\lambda
x)\leq \lambda \psi _{\ell }(x)=\lambda \sigma _{F\left( x\right) }\left(
l\right) =\sigma _{\lambda F\left( x\right) }\left( l\right) ,
\end{equation*}%
that is, $\sigma _{F\left( \lambda x\right) }\leq \sigma _{\lambda F\left(
\lambda x\right) };\ $hence, since $\ell \in \mathbb{R}_{+}^{m}\setminus
\{0\}$ is arbitrary, by Proposition \ref{properties sigma}(viii) we have $%
F(\lambda x)\subset \lambda F(x),$ which proves that $F$ is co-radiant.

(ii) For $x,x^{\prime }\in \mathbb{R}_{+}^{n}$ such that $x\leq x^{\prime }$%
, one has $F(x)\subset F(x^{\prime })$ and therefore $\psi _{\ell
}(x)=\sigma _{F(x)}(\ell )\leq \sigma _{F(x^{\prime })}(\ell )=\psi _{\ell
}(x^{\prime }).$ Conversely, assume that $F$ takes only closed normal values
and $\psi _{\ell }$ is increasing, and let $x,x^{\prime }\in \mathbb{R}%
_{+}^{n}$ be such that $x\leq x^{\prime }.$ For every $l\in \mathbb{R}%
_{+}^{n}\setminus \{0\},$ we have%
\begin{equation*}
\sigma _{F\left( x\right) }\left( l\right) =\psi _{\ell }(x)\leq \psi _{\ell
}(x^{\prime })=\sigma _{F\left( x^{\prime }\right) }\left( l\right) ,
\end{equation*}%
that is, $\sigma _{F\left( x\right) }\leq \sigma _{F\left( x^{\prime
}\right) };\ $hence, by Proposition \ref{properties sigma}(viii), we have $%
F(x)\subset F(x^{\prime }),$ which proves that $F$ is increasing.

(iii)\ Apply Lemma \ref{MaxTh} with $X:=\mathbb{R}_{+}^{n},$ $Y:=\mathbb{R}%
_{+}^{m},$ $C:=F$ and $f\left( x,y\right) :=\,\langle \ell ,y\rangle .$
\end{proof}

\bigskip

Following \cite{Rubinov}, for $k\in \mathbb{R}_{+}^{n}$ and $c\in \mathbb{R}%
_{+},$ we set%
\begin{equation*}
h_{k,c}^{+}(x):=\max \{\displaystyle\max_{i\in I}k_{i}x_{i},\,c\}.
\end{equation*}%
We further define%
\begin{equation*}
\check{h}_{k,c}^{+}(x):=\left\{ 
\begin{array}{c}
h_{k,c}^{+}(x)\text{ if }I_{+}\left( x\right) \subset I_{+}\left( k\right)
\\ 
+\infty \text{ \qquad otherwise.\qquad }%
\end{array}%
\right.
\end{equation*}%
For $l\in \mathbb{R}_{+}^{m}\setminus \{0\},$ $k\in \mathbb{R}_{+}^{n}$ and $%
c\in \mathbb{R}_{+},$ we introduce the mapping $\Delta _{l,k,c}:\mathbb{R}%
_{+}^{n}\rightrightarrows \mathbb{R}_{+}^{m}$ by%
\begin{equation*}
\Delta _{l,k,c}\left( x\right) :=\{y\in \mathbb{R}_{+}^{m}\,:\,\max \left\{
\max_{i\in I}y_{i},\langle \ell ,y\rangle \right\} \leq h_{k,c}^{+}(x)\}.
\end{equation*}

In the case when $m=1,$ it is easy to see that $\Delta _{l,k,c}=\left[ h_{%
\frac{k}{\max \left\{ l,1\right\} },\frac{c}{\max \left\{ l,1\right\} }}^{+}%
\right] .$

\begin{proposition}
\label{properties elementary}For every $l\in \mathbb{R}_{+}^{m}\setminus
\{0\},$ $k\in \mathbb{R}_{+}^{n}$ and $c\in \mathbb{R}_{+},$ the mapping $%
\Delta _{l,k,c}$ is ICR, u.s.c. and takes only nonempty compact normal
values.
\end{proposition}

\begin{proof}
Since $h_{k,c}^{+}$ is increasing and co-radiant, $\Delta _{l,k,c}$ is ICR.
Since the mapping $y\mapsto \max \left\{ \max_{i\in I}y_{i},\langle \ell
,y\rangle \right\} $ is increasing and continuous, the sets $\Delta
_{l,k,c}\left( x\right) $ are normal and closed. Since they are contained in
the cube $\left[ 0,h_{k,c}^{+}(x)\right] ^{m},$ they are also bounded, hence
compact. It is clear that they contain the origin, so they are nonempty. To
see that $\Delta _{l,k,c}$ is u.s.c., notice that it is the pointwise
intersection of the mappings $\Delta _{l,k,c}^{1}$ and $\Delta _{k,c}^{2}$
defined by%
\begin{equation*}
\Delta _{l,k,c}^{1}\left( x\right) :=\{y\in \mathbb{R}_{+}^{m}\,:\,\langle
\ell ,y\rangle \leq h_{k,c}^{+}(x)\}
\end{equation*}%
and%
\begin{equation*}
\Delta _{k,c}^{2}\left( x\right) :=\left[ 0,h_{k,c}^{+}(x)\right] ^{m},
\end{equation*}%
respectively. Since the functions $\langle \ell ,\cdot \rangle $ and $%
h_{k,c}^{+}$ are continuous, the mapping $\Delta _{l,k,c}^{1}$ has a closed
graph. The continuity of $h_{k,c}^{+}$ also implies the upper semicontinuity
of $\Delta _{k,c}^{2}.$ Hence, since $\Delta _{k,c}^{2}$ is compact-valued,
by \cite[Theorem 17.25.2]{AB06} the mapping $\Delta _{l,k,c}$ is u.s.c..
\end{proof}

\bigskip

The second main result of this paper is the following generalization of \cite%
[Theorem 5.2]{Martinez}.

\begin{theorem}
A mapping $F:\mathbb{R}_{+}^{n}\rightrightarrows \mathbb{R}_{+}^{m}$ is ICR,
u.s.c. and takes only nonempty compact normal values if, and only if, there
exists a nonempty set $T\subset \left( \mathbb{R}_{+}^{m}\setminus
\{0\}\right) \times \mathbb{R}_{+}^{n}\times \mathbb{R}_{+}$ such that%
\begin{equation}
F=\displaystyle\bigcap_{\left( l,k,c\right) \in T}\Delta _{l,k,c}.
\label{FcomoInters2da}
\end{equation}
\end{theorem}

\begin{proof}
Assume that $F$ is ICR, u.s.c. and takes only nonempty compact normal
values, and let $x\in \mathbb{R}_{+}^{n}$. By Corollary \ref{RepresSoporte},
we have%
\begin{equation}
\begin{array}{rl}
F(x) & =\{y\in \mathbb{R}_{+}^{m}\,:\,\langle \ell ,y\rangle \leq \sigma
_{F(x)}(\ell ),\;\forall \ell \in \mathbb{R}_{+}^{m}\setminus \{0\}\} \\ 
& =\{y\in \mathbb{R}_{+}^{m}\,:\,\langle \ell ,y\rangle \leq \psi _{\ell
}(x),\forall \ell \in \mathbb{R}_{+}^{m}\setminus \{0\}\} \\ 
& =\displaystyle\bigcap_{\ell \in \mathbb{R}_{+}^{m}\setminus \{0\}}\{y\in 
\mathbb{R}_{+}^{m}\;:\langle \ell ,y\rangle \leq \psi _{\ell }(x)\}.%
\end{array}
\label{F(x)=Int}
\end{equation}%
On the other hand, by Proposition \ref{ParaLasiguiente} and \cite[Theorem 5.2%
]{Martinez}, for every $\ell \in \mathbb{R}_{+}^{m}\setminus \{0\}$ there
exists a nonempty set $T_{\ell }\subset \mathbb{R}_{+}^{n}\times \mathbb{R}%
_{+}$ such that%
\begin{equation*}
\psi _{\ell }(x):=\displaystyle\inf_{(k,c)\in T_{\ell }}h_{k,c}^{+}(x).
\end{equation*}%
From this equality and (\ref{F(x)=Int}), we obtain that $F(x)=\bigcap_{\ell
\in \mathbb{R}_{+}^{m}\setminus \{0\}}\bigcap_{(k,c)\in T_{\ell }}\Delta
_{l,k,c}^{1}\left( x\right) .$ Hence $F(x)\subset
\bigcap_{i=1}^{m}\bigcap_{(k,c)\in T_{e^{i}}}\Delta _{e^{i},k,c}^{1}\left(
x\right) =\Delta _{k,c}^{2}\left( x\right) ;$ here we denote by $e^{i}$ the $%
i$-th unit vector.We therefore have%
\begin{eqnarray*}
F(x) &=&F(x)\cap \Delta _{k,c}^{2}\left( x\right) =\bigcap_{\ell \in \mathbb{%
R}_{+}^{m}\setminus \{0\}}\bigcap_{(k,c)\in T_{\ell }}\left( \Delta
_{l,k,c}^{1}\left( x\right) \cap \Delta _{k,c}^{2}\left( x\right) \right) \\
&=&\bigcap_{\ell \in \mathbb{R}_{+}^{m}\setminus \{0\}}\bigcap_{(k,c)\in
T_{\ell }}\Delta _{l,k,c}\left( x\right) .
\end{eqnarray*}%
From this equality we immediately obtain (\ref{FcomoInters2da}), with%
\begin{equation*}
T:=\left\{ \left( l,k,c\right) \in \left( \mathbb{R}_{+}^{m}\setminus
\{0\}\right) \times \mathbb{R}_{+}^{n}\times \mathbb{R}_{+}:\left(
k,c\right) \in T_{l}\right\} .
\end{equation*}

To prove the converse, assume that (\ref{FcomoInters2da}) holds for some
nonempty set $T\subset \left( \mathbb{R}_{+}^{m}\setminus \{0\}\right)
\times \mathbb{R}_{+}^{n}\times \mathbb{R}_{+}$. Then, using Proposition \ref%
{properties elementary}, it is easy to prove that the mapping $F$ is ICR and
takes only nonempty compact normal values. Its upper semicontinuity follows
from Proposition \ref{properties elementary} and \cite[Theorem 17.25.3]{AB06}%
.
\end{proof}

\bigskip

\textbf{Acknowledgments. }Abelardo Jord\'{a}n was supported by the
Departamento de Ciencias de la Pontificia Universidad Cat\'{o}lica del Per%
\'{u}. Juan Enrique Mart\'{\i}nez-Legaz was partially supported by the
Severo Ochoa Programme for Centres of Excellence in R\&D [SEV-2015-0563]. He
is affiliated with MOVE (Markets, Organizations and Votes in Economics).
Part of this work was elaborated during several visits he made to the
Escuela de Posgrado de la Pontificia Universidad Cat\'{o}lica del Per\'{u},
to which he is grateful for the financial support received as well as for
the warm hospitality of its members.

\end{document}